\documentclass[12pt]{article}
\usepackage{amsfonts}

\usepackage[latin1]{inputenc}
\usepackage{amsmath,amssymb}
\usepackage{latexsym}
\usepackage[
bookmarks=true,         
bookmarksnumbered=true, 
colorlinks=true, pdfstartview=FitV, linkcolor=blue, citecolor=blue,
urlcolor=blue]{hyperref}

\newtheorem{theorem}{Theorem}[section]

\newtheorem{lemma}[theorem]{Lemma}
\newtheorem{proposition}[theorem]{Proposition}

\numberwithin{equation}{section}
\def\qed{{\hfill $\square$ \bigskip}}
\def\R{{\mathbb R}}
\def\E{{{\mathbb E}\,}}

\def\N{{\mathbb N}}

\def\qed{{\hfill $\square$ \bigskip}}

\def\square{{\vcenter{\vbox{\hrule height.3pt
        \hbox{\vrule width.3pt height5pt \kern5pt
           \vrule width.3pt}
        \hrule height.3pt}}}}

\def\tlint{{- \kern-0.85em \int \kern-0.2em}}
\def\dlint{{- \kern-1.05em \int \kern-0.4em}}

\def \eref#1{\hbox{(\ref{#1})}}

\def \eref#1{\hbox{(\ref{#1})}}

\newenvironment{proof}[1][Proof]{\noindent\textit{#1.} }{\hfill \qed}

\begin{document}

\title{Limit laws for occupation times of stable processes}
\date{\empty }
\author{ David Nualart\thanks{ D. Nualart is supported by the NSF grant
DMS1208625.} \ \
and  \ Fangjun Xu\thanks{F. Xu is supported in part by the Robert Adams Fund.}\\
Department of Mathematics \\
University of Kansas \\
Lawrence, Kansas, 66045 USA}

\maketitle

\begin{abstract}
\noindent We prove two limit laws for 
functionals of one dimensional symmetric 1-stable process using the method of moments, and give a remark on Rosen's paper \cite{Rosen}.

\end{abstract}

\section{Introduction}

Let $X=\{X(t),\, t\geq 0\}$ be a symmetric $\alpha$-stable process in $\R$. The local time $L_t(x)$ of $X$ exists and is jointly continuous in $t$ and $x$ if $\alpha>1$ (see \cite{Boylan}).  For any integrable function $f:\R\to \R$, using the scaling property of $\alpha$-stable process and the continuity of the local time, one can easily obtain the following convergence in law in the space $C([0,\infty))$, as $n$ tends to infinity,
\[
\Big( n^{\frac{1-\alpha}{\alpha}}\int^{nt}_0 f(X(s))\, ds,\; t\geq 0 \Big) \overset{\mathcal{L}}{\longrightarrow } \Big( L_t(0)\int_{\R} f(x)\, dx,\; t\geq 0 \Big)
\]

Assuming that $f$ is a bounded Borel function on $\R$ with compact support and $\int_{\R} f(x)\, dx=0$, Rosen \cite{Rosen} showed 
\[
  \Big(  n^{\frac{1-\alpha}{2\alpha}}\int_{0}^{nt} f(X(s))\, ds\,,
  \ t\ge 0\Big) \ \ \overset{\mathcal{L}}{\longrightarrow } \ \
  \Big(  \sqrt{2c\langle f,f\rangle_{\alpha-1}}\, W (L_{t}(0))\,, t\ge
  0\Big)
\]
as $n$ tends to infinity, where $W$ is a real-valued Brownian motion independent of $X$, 
\[
c=\int^{\infty}_0(p_1(0)-p_1(1/s^{1/\alpha}))\, \frac{ds}{s^{1/\alpha}}
\]
with $p_1(x)$ being the probability density function of $X(1)$,
and 
\[
\langle f,f\rangle_{\alpha-1}=-\int_{\R^2} f(x)f(y)|x-y|^{\alpha-1}\, dx\, dy.
\]

We are interested in the limit theorems for the $\alpha$-stable process when $\alpha=1$ because the local time does not exist in this case.  We will show the following two limit laws.
\begin{theorem} \label{1st} Suppose that $f$ is bounded and $\int_{\R} |x f(x)|\, dx<\infty$. Then, for any $t>0$,
\[
\frac{1}{n}\int^{e^{nt}}_0 f(X(s))\, ds \overset{\mathcal{L}}{\longrightarrow } \Big(\frac{1}{\pi}\int_{\R} f(x)\, dx\Big)\, Z(t)
\]
\end{theorem}
as $n$ tends to infinity, where $Z(t)$ is an exponential random variable with parameter $t$.

\noindent
{\bf Remark} If we use the normalizing factor $\frac{1}{\log n}$, then the limiting distribution of 
\[
\frac{1}{\log n}\int^{nt}_0 f(X(s))\, ds
\]
is independent of $t$.

\begin{theorem} \label{2nd} Suppose that $f$ is bounded, $\int_{\R} |x f(x)|\, dx<\infty$ and $\int_{\R} f(x)\, dx=0$. Then, for any $t>0$,
\[
\frac{1}{\sqrt{n}}\int^{e^{nt}}_0 f(X(s))\, ds \overset{\mathcal{L}}{\longrightarrow } \Big(\frac{1}{\pi^2}\int_{\R}|\widehat{f}(x)|^2|x|^{-1}\, dx\Big)^{-\frac{1}{2}}\, \sqrt{Z(t)}\, \eta
\]
\end{theorem}
as $n$ tends to infinity, where $\widehat{f}$ is the Fourier transform of $f$ and $\eta$ is a standard normal random variable independent of $Z(t)$.

 In 1953, Kallianpur and Robbins \cite{KaRo} proved that  for any bounded and integrable function $f:\mathbb{R}^2\rightarrow \mathbb{R}$, 
\[
 \frac 1{\log n} \int_0^n f(B(s))ds  \overset{\mathcal{L}}{\longrightarrow } \Big(\frac  {1}{2\pi} \int_{\mathbb{R}^2 }f(x)\, dx\Big) Z(1)
 \]
as $n$ tends to infinity, where $B$ is a two-dimensional Brownian motion and $Z(1)$ is an exponential random variable with parameter $1$. After that, the asymptotic properties of the additive functionals of the form $\int^t_0 f(X(s))\, ds$, where $X$ is a real-valued stochastic process, received a lot of attention. The study of this problem mainly goes into two directions. One is on Markov processes and the other one on fractional Brownian motions. For general Markov processes, we refer to \cite{DaKa, bingham, PSV, KV}. For some special Markov processes, see, e.g., \cite{KaKo, biane} for Brownian motion and \cite{Rosen, PT} for $\alpha$-stable processes.  For general fractional Brownian motions, we refer to \cite{KaKos, Kon, HNX, NX1, NX2}. 

It is well known that general fractional Brownian motions are neither Markov processes nor semimartingles. So the martingale method applied  by Papanicolaou, Stroock and Varadhan in \cite{PSV} and further developed by Kipnis and Varadhan in \cite{KV} is not useful in the fractional Brownian motion case.  When proving the limit theorems for (additive) functionals of fractional Brownian motions, one often uses the method of moments. Another possible candidate is the Malliavin calculus. In \cite{NX2}, we introduced a chaining argument to obtain estimates for moments, using Fourier techniques. To the best of our knowledge, this chaining argument is brand new and very powerful for the fractional Brownian motion case. For example, it could have been applied to give another proof of Theorem 1.1 in \cite{Rosen}. 

In this paper, we will use the methodology in \cite{NX2} with some modifications to prove the second order limit law (Theorem \ref{2nd}) for the $1$-stable process and give another expression for the constant in the limiting process in Theorem 1.1 of \cite{Rosen}.

After some preliminaries in Section 2,  Section 3 is devoted to the
proof of Theorem \ref{1st}, based on the method of moments. In Section 4, we prove Theorem \ref{2nd} using the method of moments and the modified methodology in \cite{NX2}. In Section 5, we give a remark on Theorem 1.1 in \cite{Rosen}. Throughout this paper, if not mentioned otherwise, the letter $c$, 
with or without a subscript, denotes a generic positive finite
constant whose exact value is independent of $n$ and may change from
line to line. We use $\iota$ to denote $\sqrt{-1}$.

\bigskip

\section{Preliminaries}

Let $X=\{X(t),\, t\geq 0\}$ be a symmetric $\alpha$-stable process in $\R$. Then, the characteristic function of $X(t)$ is
\[
\E e^{\iota x X(t)}=e^{-t|x|^{\alpha}}
\]
for any $t\geq 0$ and $x\in\R$.  

The next lemma gives formulas for the moments of $Z(t)$ and $\sqrt{Z(t)}\,\eta$ where $Z(t)$ is an exponential random variable with parameter $t$ and $\eta$ is a standard normal random variable independent of $Z(t)$.
\begin{lemma} \label{m} For any $m\in\N$ and $t>0$, 
\[
\E[Z(t)]^m=m!\, t^m\quad\text{and}\quad \E[\sqrt{Z(t)}\,\eta]^{2m}=\frac{(2m)!\, t^m}{2^m}.
\]
\end{lemma}
\begin{proof} Using the moment generating function of the exponential distribution, we can easily obtain $\E[Z(t)]^m=m!\, t^m$.  Since $\eta$ and $Z(t)$ are independent, 
\[
\E[\sqrt{Z(t)}\,\eta]^{2m}=\E[Z(t)]^m\E[\eta]^{2m}=m!\, t^m\, (2m-1)!!=\frac{(2m)!\, t^m}{2^m}.
\]
\end{proof}

For $\alpha>1$, the local time $L_t(x)$ of $X$ exists and is jointly continuous in $t$ and $x$ (see \cite{Boylan}). The following lemma gives the expectation of $L_t(0)$.
\begin{lemma} \label{m1} For any $m\in\N$ and $t>0$, 
\[
\E(L_t(0))=\frac{\alpha}{2\pi(\alpha-1)}\, \Big(\int_{\R}e^{-|y|^{\alpha}}\,dy\Big)\, t^{1-\frac{1}{\alpha}}.
\]
\end{lemma}
\begin{proof} Using the Fourier transform,
\begin{align*}
\E(L_t(0))
&=\frac{1}{2\pi}\E\Big(\int^t_0\int_{\R} e^{\iota x X(s)}\,dx\, ds\Big)\\
&=\frac{1}{2\pi}\int^t_0\int_{\R} e^{-|x|^{\alpha} s}\,dx\, ds\\
&=\frac{1}{2\pi}\Big(\int_{\R}e^{-|y|^{\alpha}}\,dy\Big)\Big(\int^t_0 s^{-\frac{1}{\alpha}}\, ds\Big)\\
&=\frac{\alpha}{2\pi(\alpha-1)}\Big(\int_{\R}e^{-|y|^{\alpha}}\,dy\Big)\, t^{1-\frac{1}{\alpha}}.
\end{align*}
\end{proof}

Using the independent increments property of $\alpha$-stable processes, we can easily obtain the following useful formula. For any $m\in\N$, 
\begin{align} \label{incr}
\E\big[e^{\iota \sum\limits^{m}_{i=1} x_i X(s_i)}\big]=e^{-\sum\limits^m_{i=1}|\sum\limits^m_{j=i} x_j|^{\alpha}(s_i-s_{i-1})},
\end{align}
where $0=s_0<s_1<\cdots<s_m<\infty$ and $x_1,\dots, x_m\in\R$.

\bigskip

\section{Proof of Theorem \ref{1st}}

In this section we will prove Theorem \ref{1st}. Since $f$ is bounded, it suffices to show the limit law for 
\[
\frac{1}{n}\int^{e^{nt}}_1 f(X(s))\, ds.
\]
That is,
\begin{equation}
\frac{1}{n}\int^{e^{nt}}_1 f(X(s))\, ds \overset{\mathcal{L}}{\longrightarrow } \Big(\frac{1}{\pi}\int_{\R} f(x)\, dx\Big)\, Z(t) \label{1st1}
\end{equation}
as $n$ tends to infinity.

For any $t>0$, using the Fourier transform, we can write
\[
\frac{1}{n}\int^{e^{nt}}_1 f(X(s))\, ds=\frac{1}{2\pi n}\int^{e^{nt}}_1\int_{\R} \widehat{f}(x)\, e^{\iota x X(s)}\, dx\, ds.
\]

\begin{lemma} \label{diff} The difference of 
\[
\frac{1}{n}\int^{e^{nt}}_1\int_{\R} \widehat{f}(x)\, e^{\iota x X(s)}\, dx\, ds
\]
and
\[
\frac{1}{n}\int^{e^{nt}}_1\int_{|x|\leq 1} \widehat{f}(0)\, e^{\iota x X(s)}\, dx\, ds
\]
converges to zero in $L^2$ as $n$ tends to infinity.
\end{lemma}
\begin{proof} We first show that
\[
F_{n,1}:=\frac{1}{n}\int^{e^{nt}}_1\int_{|x|\geq 1} \widehat{f}(x)\, e^{\iota x X(s)}\, dx\, ds
\]
converges to zero in $L^2$ as $n$ tends to infinity.  This follows easily from the following argument.
\begin{align*}
\E[F_{n,1}]^2
&=\frac{2}{n^2} \int^{e^{nt}}_1\int^{s_2}_1\int_{|x_1|\geq 1}\int_{|x_2|\geq 1} \widehat{f}(x_1)\widehat{f}(x_2)\, e^{-|x_2|(s_2-s_1)-|x_2+x_1|s_1}\, dx\, ds\\
&=\frac{2}{n^2} \int^{e^{nt}}_1\int^{s_2}_1\int_{|y_1-y_2|\geq 1}\int_{|y_2|\geq 1} \widehat{f}(y_1-y_2)\widehat{f}(y_2)\,e^{-|y_2|(s_2-s_1)-|y_1|s_1}\, dy\, ds\\
&\leq \frac{c_1}{n^2} \int^{e^{nt}}_1\int_{|y_1-y_2|\geq 1}\int_{|y_2|\geq 1} |\widehat{f}(y_2)||y_2|^{-1}\, e^{-|y_1|s_1}\, dy\, ds_1\\
&\leq \frac{c_2}{n^2} \int^{e^{nt}}_1\int_{\R} e^{-|y_1|s_1}\, dy_1\, ds_1\\
&\leq c_3\frac{t}{n},
\end{align*}
where in the first inequality we used that $\widehat{f}$ is bounded.
We next show that
\[
F_{n,2}:=\frac{1}{n}\int^{e^{nt}}_1\int_{|x|\leq 1} (\widehat{f}(x)-\widehat{f}(0))\, e^{\iota x X(s)}\, dx\, ds
\]
converges to zero in $L^2$ as $n$ tends to infinity.  

Since $|\widehat{f}(x)-\widehat{f}(0)|<c_4\, |x|$ for all $x\in\R$,
\begin{align*}
\E[F_{n,2}]^2
&\leq \frac{c_5}{n^2} \int^{e^{nt}}_1\int^{s_2}_1\int_{|x_1|\leq 1}\int_{|x_2|\leq 1} |x_1||x_2|\, e^{-|x_2|(s_2-s_1)-|x_2+x_1|s_1}\, dx\, ds\\
&\leq \frac{c_5}{n^2} \int^{e^{nt}}_1\int^{s_2}_1\int_{|y_1-y_2|\leq 1}\int_{|y_2|\leq 1} |y_2|\, e^{-|y_2|(s_2-s_1)-|y_1|s_1}\, dy\, ds\\
&\leq \frac{c_6}{n^2} \int^{e^{nt}}_1\int_{|y_1|\leq 2} e^{-|y_1|s_1}\, dy_1\, ds_1\\
&\leq  c_7\, \frac{t}{n}. 
\end{align*}
Combining these two estimates gives the desired result.
\end{proof}

With the help of Lemma \ref{diff}, to prove Theorem \ref{1st}, we only need to show the following result.
\begin{proposition} Suppose that $f$ is bounded and $\int_{\R} |x f(x)|\, dx<\infty$. Then, for any $t>0$,
\[
\frac{1}{2\pi n}\int^{e^{nt}}_1\int_{|x|\leq 1} \widehat{f}(0)\, e^{-\iota x X(s)}\, dx\, ds \overset{\mathcal{L}}{\longrightarrow } \Big(\frac{1}{\pi}\int_{\R} f(x)\, dx\Big)\, Z(t)
\]
as $n$ tends to infinity, where $Z(t)$ is an exponential random variable with parameter $t$.
\end{proposition}
\begin{proof}  Let 
\[
F_n=\frac{1}{2\pi n}\int^{e^{nt}}_1\int_{|x|\leq 1} \widehat{f}(0)\, e^{\iota x X(s)}\, dx\, ds.
\] 
We first show tightness. Note that
\begin{align*}
\E[F_n]^2
&=\frac{(\widehat{f}(0))^2}{2\pi^2 n^2} \int^{e^{nt}}_1\int^{s_2}_1\int_{[-1,1]^2} e^{-|x_2|(s_2-s_1)-|x_2+x_1|s_1}\, dx\, ds\\
&=\frac{(\widehat{f}(0))^2}{2\pi^2 n^2} \int^{e^{nt}}_1\int^{s_2}_1\int_{|y_2|\leq 1}\int_{|y_1-y_2|\leq 1} e^{-|y_2|(s_2-s_1)-|y_1|s_1}\, dy\, ds.
\end{align*}
Integrating with respect to the variable $y$ and taking into account that $|y_1-y_2|\leq 1$ and $|y_2|\leq 1$ implies $|y_1|\leq 2$,
\begin{align*}
\E[F_n]^2
&\leq \frac{2(\widehat{f}(0))^2}{\pi^2 n^2} \int^{e^{nt}}_1\int^{s_2}_1 \frac{1-e^{-(s_2-s_1)}}{s_2-s_1}\frac{1-e^{-2s_1}}{s_1} \, ds.
\end{align*}
Making the change of variables $u_2=s_2-s_1$ and $u_1=2s_1$, we can write
\begin{align*}
\E[F_n]^2
&\leq \frac{2(\widehat{f}(0))^2}{\pi^2 n^2} \int^{e^{2nt}}_0\int^{e^{2nt}}_0 \frac{1-e^{-u_2}}{u_2}\frac{1-e^{-u_1}}{u_1} \, du\\
&= \frac{2(\widehat{f}(0))^2}{\pi^2} \Big(\frac{1}{n}\int^{2e^{nt}}_0\frac{1-e^{-v}}{v} \, dv\Big)^2\\
&\leq \frac{2(\widehat{f}(0))^2}{\pi^2} \Big(\frac{1}{n}\int^{1}_0 1\,  dv+\frac{1}{n}\int^{2e^{nt}}_1\frac{1}{v} \, dv\Big)^2\\
&\leq c\, (\widehat{f}(0))^2 (1+t)^2.
\end{align*}
We next show the convergence of all moments. For any $m\in\N$, let $I^n_m=\E\big(F_n\big)^m$. Then
\begin{align*}
I^n_m
&=\Big(\frac{\widehat{f}(0)}{2\pi}\Big)^m \frac{1}{n^m}\,\E\Big(\int^{e^{nt}}_1\int_{|x|\leq 1} \, e^{\iota x X(s)}\, dx\, ds \Big)^m\\
&=\Big(\frac{\widehat{f}(0)}{2\pi}\Big)^m \frac{m! }{n^m}\int_{[-1,1]^m} \int_{D_m}\, e^{-\sum\limits^m_{i=1} |\sum\limits^m_{j=i}x_j|(s_i-s_{i-1})}\, ds\, dx,
\end{align*}
where $D_m=\big\{1<s_1<\cdots<s_m<e^{nt} \big\}$, with the convention $s_0=0$.

By Lemma \ref{a1},
\begin{align*}
\lim_{n\to\infty} I^n_m
&=\Big(\frac{\widehat{f}(0)}{2\pi}\Big)^m\, m!\, \lim_{n\to\infty} \frac{1}{n^m}\int_{[-1,1]^m} \int_{D_m}\, e^{-\sum\limits^m_{i=1} |\sum\limits^m_{j=i}x_j|(s_i-s_{i-1})}\, ds\, dx\\
&=\Big(\frac{\widehat{f}(0)}{2\pi}\Big)^m\, m!\, (2t)^m\\
&=\Big(\frac{1}{\pi}\int_{\R} f(x)\, dx\Big)^m\, \E[Z(t)]^m.
\end{align*}
Using the method of moments, the proof is completed.
\end{proof}

\bigskip

{\medskip \noindent \textbf{Proof of Theorem \ref{1st}.}}  This follows
from Lemma \ref{diff}, Proposition \ref{1st1}, and the argument before Lemma \ref{diff}.

\bigskip

\section{Proof of Theorem \ref{2nd}}
In this section, we shall show Theorem \ref{2nd}. Since $f$ is bounded,  we only need to consider the convergence of the following random variables
\[
F_n=\frac{1}{\sqrt{n}}\int^{e^{nt}}_1 f(X(s))\, ds.
\]

For $m\in\N$, let
\[
I^n_m= \frac{m!}{n^{\frac{m}{2}}}\, \E\Big[\int_{D_{m,1}} \Big(\prod^m_{i=1}f(X(s_i))\Big)\, ds\Big],
\]
where $D_{m,1}=\big\{(s_1, \dots, s_m) \in D_m: s_i-s_{i-1}\geq n^{-m},\, i=2,3,\dots, m\big\}$ and $D_m=\big\{1<s_1<\cdots<s_m<e^{nt} \big\}$ as before.
Then,   taking into account that $f$ is bounded, we can write
\begin{align*}
\Big|\E(F_n)^m-I^n_m\Big|  & \le \frac{m!}{n^{\frac{m}{2}}}\, \sum_{j=1}^m \E\Big[\int_{D_{m}\cap\{|s_j -s_{j-1}| <n^{-m}\}} \Big(\prod^m_{i=1}|f(X(s_i))|\Big)\, ds\Big]\\
& \le \|f\|_\infty\, \frac{m m!}{n^{\frac{3m}{2}}}\,  \E\Big[\int_{D_{m-1}} \Big(\prod^{m-1}_{i=1}|f(X(s_i))|\Big)\, ds\Big].
\end{align*}
Thus,  Theorem \ref{1st} implies that
\begin{equation} \label{eq1}
  \Big|\E(F_n)^m-I^n_m\Big|\le c_1 n^{-\frac m2 -1}.
\end{equation}
Applying the Fourier transform, we can write
\begin{align*}
I^n_m
&=\frac{m!}{(2\pi \sqrt{n})^m}\int_{\R^m}\int_{D_{m,1}} \Big(\prod^m_{i=1}\widehat{f}(x_i)\Big)\,\E\Big(e^{\iota \sum\limits^m_{i=1} x_i X(s_i)}\Big)\, ds\, dx.
\end{align*}
Using  \eref{incr} and then making the change of variables $y_i=\sum\limits^m_{j=i}x_j$ for $i=1,2,\dots, m$ gives
\begin{align*}
I^n_m
&=\frac{m!}{(2\pi  \sqrt{n})^m}\int_{\R^m}\int_{D_{m,1}} \Big(\prod^m_{i=1}\widehat{f}(y_i-y_{i+1})\Big)\, e^{-\sum\limits^{m}_{i=1} |y_i| (s_i-s_{i-1})}\, ds\, dy.
\end{align*}

Let $I^n_{m,0}=I^n_m$. For $k=1,\dots,m$, we define
\begin{align*}
I^n_{m,k}
&=\frac{m!}{(2\pi  \sqrt{n})^m}\int_{\R^m}\int_{D_{m,1}}  I_k\,  \prod^{m}_{i=k+1} \widehat{f}(y_i-y_{i+1})\, e^{-\sum\limits^{m}_{i=1} |y_i| (s_i-s_{i-1})}\, ds\, dy,
\end{align*}
where 
\[
I_k =
\begin{cases}
\prod\limits^{\frac{k-1}{2}}_{j=1}|\widehat{f}(y_{2j})|^2 \widehat{f}(-y_{k+1}), & \text{if } k \text{ is odd}; \\
\prod\limits^{\frac{k}{2}}_{j=1}|\widehat{f}(y_{2j})|^2, & \text{if }k\text{ is even}.
\end{cases}
\]

The following proposition, which is similar to Proposition 3.1 in \cite{NX2}, controls the difference between $I^n_{m,k-1}$ and $ I^n_{m,k}$. We fix a positive constant $\gamma$ strictly less than $\frac{1}{2}$.

\begin{proposition} \label{chain} 
For $k=1,2,\dots,m$, there exists a positive constant $c$, which depends on $\gamma$, such that
\[
|I^n_{m,k-1}-I^n_{m,k}|\leq c\, n^{-\gamma}.
\]
\end{proposition}
 
\begin{proof} The proof will be done in several steps.

\noindent \textit{Step 1.} Suppose first that $k=1$. Making the change of variables $u_1=s_1$, $u_i =s_i -s_{i-1}$, for $2\le i\le m$, we can show that $|I^n_{m,0}-I^n_{m,1}|$ is less than a constant multiple of
\[
n^{-\frac{m}{2}}\int_{\R^m}\int_{O_m} \big|\widehat{f}(y_1-y_2)-\widehat{f}(-y_2)\big| \Big(\prod^m_{i=2}  |\widehat{f}(y_i-y_{i+1})|\Big)\, e^{-\sum\limits^m_{i=1} |y_i|u_i}\, du\, dy,
\]
where $O_m=\big\{(u_1, \dots, u_m):  u_1>1, \sum_{i=1}^m u_i<e^{nt},  n^{-m}<u_i,\, i=2,3,\dots, m \big\}$ and $y_{m+1}=0$.

Taking into account that  that $|\widehat{f}(x)|\leq c_{\alpha}(|x|^{\alpha}\wedge 1)$ for $\alpha\in[0,1]$, and $O_m \subset [n^{-m}, e^{nt}]^m$,  we obtain
\begin{align*}
|I^n_{m,0}-I^n_{m,1}|
&\leq c_2\, n^{-\frac{m}{2}}\int_{\R^m}\int_{[n^{-m}, e^{nt}]^m} |y_1|^{\alpha}\prod^{\lfloor\frac{m}{2}\rfloor }_{j=1} \big(|y_{2j}|^{\alpha}+|y_{2j+1}|^{\alpha}\big) e^{-\sum\limits^m_{i=1} |y_i| u_i}\, du\, dy\\
&\leq  c_3\, n^{-\frac{m}{2}+(\lfloor\frac{m}{2}\rfloor+1) (m\alpha)+(m-1-\lfloor\frac{m}{2}\rfloor)}\\
&\leq  c_3\, n^{-\frac{1}{2}+(\lfloor\frac{m}{2}\rfloor+1) (m\alpha)}. 
\end{align*}
Choosing $\alpha$ small enough such that $-\frac{1}{2}+(\lfloor\frac{m}{2}\rfloor+1) (m\alpha)=-\gamma$ gives
\begin{align*}
|I^n_{m,0}-I^n_{m,1}| \leq  c_3\, n^{-\gamma}. 
\end{align*}

\noindent \textit{Step 2:} Suppose now that $k=2$. 
By the definition of $I^n_{m,1}$ and $I^n_{m,2}$, $|I^n_{m,1}-I^n_{m,2}|$ is less than a constant multiple of
\[
n^{-\frac{m}{2}}\int_{\R^m}\int_{ [n^{-m}, e^{nt}]^m} \big|\widehat{f}(-y_2)\big|\big|\widehat{f}(y_2-y_3)-\widehat{f}(y_2)\big| \Big(\prod^m_{i=3}  |\widehat{f}(y_i-y_{i+1})|\Big)\, e^{-\sum\limits^m_{i=1} |y_i|\, u_i}\, du\, dy.
\]
Using similar arguments as in Step 1,
 \begin{align*}
 |I^n_{m,1}-I^n_{m,2}|
&\leq c_4\, n^{-\frac{m}{2}} \int_{\R^m}\int_{ [n^{-m}, e^{nt}]^m} |y_2|^{\alpha}|y_3|^{\alpha} \Big(\prod^{\lfloor \frac{m}{2}\rfloor}_{j=2}  |y_{2j}|^{\alpha}+|y_{2j+1}|^{\alpha}\Big)\, e^{-\sum\limits^m_{i=1} |y_i|\, u_i}\, du\, dy\\
&\leq c_5\, n^{-\frac{m}{2}+(\lfloor \frac{m}{2}\rfloor+1)(m\alpha)+(m-1-\lfloor \frac{m}{2}\rfloor)}\\
&\leq c_5\, n^{-\frac{1}{2}+(\lfloor \frac{m}{2}\rfloor+1)(m\alpha)}\\
&= c_5\, n^{-\gamma}.
\end{align*}

\noindent \textit{Step 3:}
 Suppose that  $k$ is odd and $3\leq k\leq m$.  Since $k$ is odd, $|I^n_{m,k-1}-I^n_{m,k}|$ is less than a constant multiple of
\begin{align*}
&n^{-\frac{m}{2}}\int_{\R^m}\int_{ [n^{-m}, e^{nt}]^m} \Big(\prod^m_{i=k+1}  |\widehat{f}(y_i-y_{i+1})|\Big)\, \big|\widehat{f}(y_k-y_{k+1})-\widehat{f}(-y_{k+1})\big| \\
&\qquad\qquad\times \Big(\prod^{\frac{k-1}{2}}_{j=1}|\widehat{f}(y_{2j})|^2\Big)\, e^{-\sum\limits^m_{i=1} |y_i|\, u_i}\, du\, dy.
\end{align*}
Therefore,   $|I^n_{m,k-1}-I^n_{m,k}|$ is less than a constant multiple of
\[
n^{-\frac{m}{2}}\int_{\R^m}\int_{O_m}  \Big(\prod^m_{i=k+1}  |\widehat{f}(y_i-y_{i+1})|\Big)  |y_{k}|^{\alpha}\Big(\prod^{\frac{k-1}{2}}_{j=1}|\widehat{f}(y_{2j})|^2\Big) e^{-\sum\limits^m_{i=1} |y_i|\, u_i}\, du\, dy.
\]
Integrating with respect to the variables $x_i$s and $u_i$s with $i\leq k-1$ gives
\begin{align*}
|I^n_{m,k-1}-I^n_{m,k}|
&\leq c_6\, n^{-\frac{m-(k-1)}{2}}\int_{\R^{m-k+1}}\int_{[n^{-m}, e^{nt}]^{m-k+1}}  \Big(\prod^m_{i=k+1}  |\widehat{f}(y_i-y_{i+1})|\Big) |y_k|^{\alpha}\\
&\qquad\qquad\times e^{-\sum\limits^m_{i=k} |y_i|\, u_i}\, du\, dy,
\end{align*}
where $du=du_k\cdots du_{m}$, $dy=dy_k\cdots dy_{m}$. Applying Step 1 and then doing some algebra, we can obtain
\[
|I^n_{m,k-1}-I^n_{m,k}|\leq c_6\, n^{-\frac{1}{2}+(\lfloor\frac{m-k+1}{2}\rfloor+1)(m-k+1)\alpha}=c_6\, n^{-\gamma}.
\]
\noindent \textit{Step 4:} The case $k$ is even  and $4\leq k\leq m$ is handled in a similar way.
\end{proof}

\begin{proposition} \label{2nd1} Suppose that $f$ is bounded, $\int_{\R} |x f(x)|\, dx<\infty$ and $\int_{\R} f(x)\, dx=0$. Then, for any $t>0$,
\[
\frac{1}{\sqrt{n}}\int^{e^{nt}}_1 f(X(s))\, ds \overset{\mathcal{L}}{\longrightarrow }\Big( \frac{1}{\pi^2}\int_{\R}|\widehat{f}(x)|^2|x|^{-1}\, dx\Big)^{-\frac{1}{2}}\, \sqrt{Z(t)}\, \eta
\]
as $n$ tends to infinity, where $\eta$ is a standard normal random variable independent of $Z(t)$.
\end{proposition}
\begin{proof}  The proof will be done in several steps.

\medskip \noindent
\textbf{Step 1} \quad  We first show tightness.  Let $F_n=\frac{1}{\sqrt{n}}\int^{e^{nt}}_1 f(X(s))\, ds$. Then, using the Fourier transform,
\begin{align*}
\E(F_n)^2
&=\frac{2}{n}\int^{e^{nt}}_1\int^{s_2}_1\int_{\R^2} \widehat{f}(x_1)\widehat{f}(x_2)\, e^{-|x_2|(s_2-s_1)-|x_2+x_1|s_1}\, dx\, ds.
\end{align*}
Since $|\widehat{f}(x)|\leq c_{\alpha} (|x|^{\alpha}\wedge 1)$ for all $x\in\R$ and $\alpha\in[0,1]$,
\begin{align*}
\E(F_n)^2
&\leq \frac{c_1}{n}\int^{e^{nt}}_1\int^{s_2}_1\int_{\R^2} |\widehat{f}(x_2)|\, e^{-|x_2|(s_2-s_1)-|x_2+x_1|s_1}\, dx\, ds\\
&\leq \frac{c_2}{n}\Big(\int^{e^{nt}}_1s^{-1}_1\, ds_1\Big)\Big(\int_{\R} |\widehat{f}(x_2)||x_2|^{-1}\, dx_2\Big)\\
&\leq c_3\, t.
\end{align*}

\medskip \noindent
\textbf{Step 2} \quad We show the convergence of all odd moments.  Assume that $m$ is odd. Recall the estimate  (\ref{eq1}), which allows us to replace 
$\E(F_n)^{m}$ by $I^n_{m}$. 
By Proposition \ref{chain}, it suffices to show 
\[
\lim_{n\to\infty} I^n_{m,m}=0,
\]
where
\[
I^n_{m, m}=\frac{m!}{(2\pi  \sqrt{n})^m}\int_{\R^m}\int_{D_{m,1}}  \widehat{f}(y_m)  \prod^{\frac{m-1}{2}}_{j=1} |\widehat{f}(y_{2j})|^2\, e^{-\sum\limits^{m}_{i=1} |y_i| (s_i-s_{i-1})}\, ds\, dy.
\]
Making the change of variables $u_1=s_1$, $u_i =s_i -s_{i-1}$, for $2\le i\le m$ yields
\[
I^n_{m, m}=\frac{m!}{(2\pi  \sqrt{n})^m}\int_{\R^m}\int_{O_{m}}  \widehat{f}(y_m)  \prod^{\frac{m-1}{2}}_{j=1} |\widehat{f}(y_{2j})|^2\, e^{-\sum\limits^{m}_{i=1} |y_i| u_i}\, du\, dy,
\]
where, as before,
\[
O_{m}=\big\{(u_1, \dots, u_{m}): \, 1<u_1, \sum_{i=1}^m u_i <e^{nt}, n^{-m}<u_i<e^{nt},\, i=2,\dots, m \big\}.
\]
Notice that $O_m \subset [1,e^{nt}] \times [n^{-m}, e^{nt}]^{m-1}$. As a consequence,
\begin{align*}
\big| I^n_{m, m}\big|
&\leq c_4\, n^{-\frac{m}{2}}\int_{\R^m}\int_{[1,e^{nt}] \times [n^{-m}, e^{nt}]^{m-1}}  |\widehat{f}(y_m)|\prod^{\frac{m-1}{2}}_{j=1} |\widehat{f}(y_{2j})|^2\, e^{-\sum\limits^{m}_{i=1} |y_i| u_i}\, du\, dy\\
&\leq c_5\,n^{-\frac{m}{2}}\, \Big(\int_{\R} |\widehat{f}(y)|^2|y|^{-1}dy\Big)^{\frac{m-1}{2}} \Big(\int^{e^{nt}}_{n^{-m}} u^{-1}\, du\Big)^{\frac{m-1}{2}}\, \Big(\int_{\R}|\widehat{f}(y)|\, |y|^{-1} dy\Big)\\
&\leq c_6\, n^{-\frac{1}{2}}.
\end{align*}
Combining these estimates gives $\lim\limits_{n\to\infty}\E(F_n)^{m}=0$ when $m$ is odd.

\medskip \noindent
\textbf{Step 3} \quad We show the convergence of all even moments. Assume that $m$ is even. Recall  the estimate (\ref{eq1}). 
By Proposition \ref{chain}, it suffices to show 
\begin{align} \label{2nd1-0}
\lim_{n\to\infty} I^n_{m,m}=\Big( \frac{1}{2\pi^2}\int_{\R}|\widehat{f}(x)|^2|x|^{-1}\, dx\Big)^{-\frac{m}{2}}\, \E\Big(\sqrt{Z(t)}\, \eta\Big)^{m},
\end{align}
where
\[
I^n_{m, m}=\frac{m!}{(2\pi  \sqrt{n})^m}\int_{\R^m}\int_{D_{m,1}} \Big(\prod^{m/2}_{j=1} |\widehat{f}(y_{2j})|^2\Big)\, e^{-\sum\limits^{m}_{i=1} |y_i| (s_i-s_{i-1})}\, ds\, dy.
\]

Making the change of variables $u_1=s_1$, $u_i =s_i -s_{i-1}$, for $2\le i\le m$ and then integrating with respect to all $y_i$s with $i$ odd gives
\begin{align*}
I^n_{m,m}
&=\frac{2^{\frac{m}{2}}m!}{(2\pi\sqrt{n})^{m}}\int_{\R^{\frac{m}{2}}}\int_{O_{m}}  \Big(\prod^{m/2}_{j=1}| \widehat{f}(y_{2j})|^2\Big)\, e^{-\sum\limits^{m/2}_{j=1} |y_{2j}| u_{2j}}
\Big(\prod^{m/2}_{j=1} u^{-1}_{2j-1}\Big)\, du\, d\overline{y},
\end{align*}
where $d\overline{y}=dy_2\, dy_4\,\cdots dy_m$ and, as before,
\[
O_{m}=\big\{(u_1, \dots, u_{m}): \, 1<u_1, \sum_{i=1}^m u_i <e^{nt}, n^{-m}<u_i<e^{nt},\, i=2,\dots, m \big\}.
\]
Taking into account that $O_m \subset [1,\infty) \times [n^{-m}, e^{nt}]^{m-1}$ yields
\begin{align} \label{2nd1-1} \nonumber
\limsup_{n\to\infty} I^n_{m,m}
&\leq\lim_{n\to\infty}\frac{2^{\frac{m}{2}}m!}{(2\pi)^{m}}\Big(\int_{\R}\int^{\infty}_0|\widehat{f}(y)|^2\, e^{-|y| u}\, du\, dy\Big)^{\frac{m}{2}}\Big(\frac{1}{n}\int^{e^{nt}}_{n^{-m}}u^{-1} du\Big)^{\frac{m}{2}}\\ 
&= \frac{m!\ t^\frac{m}{2}}{2^{\frac{m}{2}}} \Big(\frac{1}{\pi^2}\int_{\R} | \widehat{f}(y)|^2|y|^{-1}\, dy\Big)^{\frac{m}{2}}.
\end{align}
On the other hand, using $O_{m,1}\times O_{m,2} \subset O_m$, where
\[
O_{m,1}=\Big\{(u_1,\dots, u_{m-1}): u_1>1,\, u_{2j-1}>n^{-m},\, j=2,\dots,m/2,\, \sum^{m/2}_{j=1} u_{2j-1}<e^{nt}/2\Big\}
\]
and
\[
O_{m,2}=\Big\{(u_2,\dots, u_{m}): \, u_{2j}>n^{-m},\, j=1,\dots,m/2,\, \sum^{m/2}_{j=1} u_{2j}<e^{nt}/2\Big\},
\]
gives
\begin{align*}
I^n_{m,m}
&\geq \frac{2^{\frac{m}{2}} m!}{(2\pi\sqrt{n})^{m}} \int_{\R^m}\int_{O_{m,1}\times O_{m,2}}  \Big(\prod^{\frac{m}{2}}_{j=1}| \widehat{f}(y_{2j})|^2\Big)\,e^{-\sum\limits^{m/2}_{j=1} |y_{2j}| u_{2j}}\Big(\prod^{m/2}_{j=1} u^{-1}_{2j-1}\Big)\, du\, d\overline{y}.
\end{align*}
 By Lemmas \ref{a2} and \ref{a3},
\begin{align} \label{2nd1-2} \nonumber
\liminf_{n\to\infty} I^n_{m,m}
& \geq \frac{2^{\frac{m}{2}} m!}{(2\pi)^{m}} \Big(\int_{\R} |\widehat{f}(y)|^2|y|^{-1}\, dy\Big)^\frac{m}{2} t^{\frac{m}{2}}\\
&= \frac{m!\ t^{\frac{m}{2}}}{2^{\frac{m}{2}}} \Big(\frac{1}{\pi^2}\int_{\R} | \widehat{f}(y)|^2|y|^{-1}\, dy\Big)^{\frac{m}{2}}.
\end{align}

Combining \eref{2nd1-1} and \eref{2nd1-2} gives
\begin{align*}
\lim_{n\to\infty} I^n_{m,m}
&= \frac{m!\ t^{\frac{m}{2}}}{2^{\frac{m}{2}}} \Big(\frac{1}{\pi^2}\int_{\R} | \widehat{f}(y)|^2|y|^{-1}\, dy\Big)^{\frac{m}{2}}\\
&=\Big( \frac{1}{\pi^2}\int_{\R}|\widehat{f}(x)|^2|x|^{-1}\, dx\Big)^{-\frac{m}{2}}\, \E\Big(\sqrt{Z_t}\, \eta\Big)^{m},
\end{align*}
where in the last equality we used Lemma \ref{m}. So the statement \eref{2nd1-0} follows. Using the method of moments, the proof is completed.
\end{proof}

\bigskip

{\medskip \noindent \textbf{Proof of Theorem \ref{2nd}.}} Since $f$ is bounded, this follows easily from Proposition \ref{2nd1}.

\bigskip

\section{A remark on \cite{Rosen}}
In this section, we assume $\alpha>1$ and will give another expression for the constant in the limiting process in Theorem 1.1 of \cite{Rosen}.
\begin{theorem}
Suppose that $f$ is a bounded Borel function on $\R$ with compact support and $\int_{\R} f(x)\, dx=0$. Then
\[
  \bigg\{ n^{\frac{1-\alpha}{2\alpha}}\int_{0}^{nt} f(X(s))\, ds\,,
  \ t\ge 0\bigg\} \ \ \overset{\mathcal{L}}{\longrightarrow } \ \
  \bigg\{\Big(\frac{1}{\pi^2}\int_{\R}|\widehat{f}(x)|^2|x|^{-\alpha}\, dx\Big)^{-\frac{1}{2}}\, W (L_{t}(0)\,, t\ge
  0\bigg\}
\]
as $n$ tends to infinity, where $W$ is a real-valued Brownian motion independent of $X$.
\end{theorem}
\begin{proof} Let 
\[
F_n= n^{\frac{1-\alpha}{2\alpha}}\int_{0}^{nt} f(X(s))\, ds.
\]
Then, by Theorem 1.1 in \cite{Rosen}, it suffices to show
\[
\lim_{n\to\infty}\E(F_n)^2=\Big(\frac{1}{\pi^2}\int_{\R}|\widehat{f}(x)|^2|x|^{-\alpha}\, dx\Big)\, \E(L_t(0)).
\] 
Using the Fourier transform,
\begin{align*}
\E(F_n)^2&=\frac{1}{2\pi^2}\, n^{\frac{1-\alpha}{\alpha}}\int_{0}^{nt}\int^{s_2}_0\int_{\R^2}\widehat{f}(x_1)\widehat{f}(x_2)\, e^{-|x_2|^{\alpha}(s_2-s_1)-|x_2+x_1|^{\alpha}s_1 }\, dx\, ds.
\end{align*}
Making the change of variables $y_2=x_2$ and $y_1=x_2+x_1$,
\begin{align*}
\E(F_n)^2&=\frac{1}{2\pi^2}\, n^{\frac{1-\alpha}{\alpha}}\int_{0}^{nt}\int^{s_2}_0\int_{\R^2}\widehat{f}(y_1-y_2)\widehat{f}(y_2)\, e^{-|y_2|^{\alpha}(s_2-s_1)-|y_1|^{\alpha}s_1 }\, dy\, ds.
\end{align*}
Let 
\[
I^2_2=\frac{1}{2\pi^2}\, n^{\frac{1-\alpha}{\alpha}}\int_{0}^{nt}\int^{s_2}_0\int_{\R^2}|\widehat{f}(y_2)|^2\, e^{-|y_2|^{\alpha}(s_2-s_1)-|y_1|^{\alpha}s_1 }\, dy\, ds.
\] 

For all $x,y\in\R$, 
\[
|\widehat{f}(x)-\widehat{f}(y)|\leq c_{\beta}|x-y|^{\beta},
\]  
where $\beta$ can be any constant in $[0,1]$. Thus,
\begin{align*}
\big|\E(F_n)^2-I^2_2\big|
&\leq c_1\, n^{\frac{1-\alpha}{\alpha}}\int_{0}^{nt}\int^{s_2}_0\int_{\R^2}|y_1|^{\beta}|\widehat{f}(y_2)|\, e^{-|y_2|^{\alpha}(s_2-s_1)-|y_1|^{\alpha}s_1 }\, dy\, ds.
\end{align*}
Making the change of variables $u_2=s_2-s_1$ and $u_1=s_2$,
\begin{align*}
\big|\E(F_n)^2-I^2_2\big|
&\leq c_1\, n^{\frac{1-\alpha}{\alpha}}\int_{0}^{nt}\int^{nt}_0\int_{\R^2}|y_1|^{\beta}|\widehat{f}(y_2)|\, e^{-|y_2|^{\alpha}u_2-|y_1|^{\alpha}u_1 }\,dy\,du\\
&\leq c_1\, n^{\frac{1-\alpha}{\alpha}}\int_{0}^{nt}\int_{\R^2}|y_1|^{\beta}|\widehat{f}(y_2)| |y_2|^{-\alpha}\, e^{-|y_1|^{\alpha}u_1 }\, dy\, du_1\\
&\leq c_2\, n^{\frac{1-\alpha}{\alpha}}\int_{0}^{nt}\int_{\R}|y_1|^{\beta}\, e^{-|y_1|^{\alpha}u_1 }\, dy_1\, du_1\\
&\leq c_3\, n^{\frac{1-\alpha}{\alpha}}\int_{0}^{nt}|u_1|^{-\frac{1+\beta}{\alpha}}\, du_1\\
&\leq c_4\, n^{-\frac{\beta}{\alpha}},
\end{align*}
where $\beta$ can be any constant such that $0<\beta<\alpha-1$. Using the above estimate, 
\[
\lim_{n\to\infty} \E(F_n)^2=\lim_{n\to\infty} I^2_2. 
\]
So we only need to show 
\[
\lim\limits_{n\to\infty} I^2_2=\Big(\frac{1}{\pi^2}\int_{\R}|\widehat{f}(x)|^2|x|^{-\alpha}\, dx\Big)\, \E(L_t(0)).
\] 

By the L'H\^opital rule,
\begin{align*}
\lim_{n\to\infty} I^2_2
&=\frac{\alpha t}{2\pi^2(\alpha-1)}\lim_{n\to\infty}\, n^{\frac{1}{\alpha}}\int_{0}^{nt}\int_{\R^2}|\widehat{f}(y_2)|^2\, e^{-|y_2|^{\alpha}(nt-s_1)-|y_1|^{\alpha}s_1 }\, dy\, ds_1\\
&=\frac{\alpha t}{2\pi^2(\alpha-1)}\Big(\int_{\R}e^{-|y|^{\alpha}}\, dy\Big)\lim_{n\to\infty}\, n^{\frac{1}{\alpha}}\int_{0}^{nt}\int_{\R}|\widehat{f}(y)|^2\, s^{-\frac{1}{\alpha}}e^{-|y|^{\alpha}(nt-s)}\, dy\, ds.
\end{align*}

For any $\epsilon\in(0,1)$, we obtain
\begin{align*}
&\lim_{n\to\infty}\, n^{\frac{1}{\alpha}}\int_{0}^{(1-\epsilon)nt}\int_{\R}|\widehat{f}(y)|^2\, s^{-\frac{1}{\alpha}}e^{-|y|^{\alpha}(nt-s)}\, dy\, ds\\
&\leq\lim_{n\to\infty}\, n^{\frac{1}{\alpha}}\int_{0}^{(1-\epsilon)nt}\int_{\R}|\widehat{f}(y)|^2\, s^{-\frac{1}{\alpha}}e^{-|y|^{\alpha}\epsilon nt}\, dy\, ds\\
&\leq c_5\, \lim_{n\to\infty}\, n \int_{\R}|y|^2\, e^{-|y|^{\alpha}\epsilon nt}\, dy\\
&\leq c_6\, \lim_{n\to\infty}\, n^{1-\frac{3}{\alpha}}\\
&=0,
\end{align*}
where in the second inequality we used $|\widehat{f}(y)|\leq c_7|y|$ for all $y\in\R$.

Making the change of variable $u=nt-s$,
\begin{align*}
&\lim_{n\to\infty}\, n^{\frac{1}{\alpha}}\int^{nt}_{(1-\epsilon)nt}\int_{\R}|\widehat{f}(y)|^2\, s^{-\frac{1}{\alpha}}\, e^{-|y|^{\alpha}(nt-s)}\, dy\, ds\\
&\qquad\qquad=\lim_{n\to\infty}\, \int_{0}^{\epsilon nt}\int_{\R}|\widehat{f}(y)|^2\, (t-\frac{u}{n})^{-\frac{1}{\alpha}}\, e^{-|y|^{\alpha} u}\, dy\, du.
\end{align*}
Note that
\[
t^{-\frac{1}{\alpha}}\leq (t-\frac{u}{n})^{-\frac{1}{\alpha}}\leq (1-\epsilon)^{-\frac{1}{\alpha}}\, t^{-\frac{1}{\alpha}}.
\]
This gives
\begin{align*}
& t^{-\frac{1}{\alpha}}\, \int_{\R} |\widehat{f}(y)|^2|y|^{-\alpha}\, dy\\
&\qquad\leq \lim_{n\to\infty}\, n^{\frac{1}{\alpha}}\int^{nt}_{(1-\epsilon)nt}\int_{\R}|\widehat{f}(y)|^2\, s^{-\frac{1}{\alpha}}\, e^{-|y|^{\alpha}(nt-s)}\, dy\, ds\\
&\qquad\qquad \leq (1-\epsilon)^{-\frac{1}{\alpha}}\,  t^{-\frac{1}{\alpha}}\, \int_{\R} |\widehat{f}(y)|^2|y|^{-\alpha}\, dy.
\end{align*}

Using these estimates and the fact that $\epsilon>0$ is arbitrary, we obtain
\begin{align*}
\lim\limits_{n\to\infty} I^2_2
&=\frac{\alpha }{2\pi^2(\alpha-1)}\Big(\int_{\R}e^{-|y|^{\alpha}}\, dy\Big)\, t^{1-\frac{1}{\alpha}}\, \int_{\R} |\widehat{f}(y)|^2|y|^{-\alpha}\, dy\\
&=\Big(\frac{1}{\pi}\int_{\R} |\widehat{f}(y)|^2|y|^{-\alpha}\, dy\Big)\, \E(L_t(0)),
\end{align*}
where in the last equality we used Lemma \ref{m1}. This completes the proof.
\end{proof}

\bigskip

\section{Appendix}

Here we give some lemmas which are necessary for the proof of Theorems \ref{1st} and \ref{2nd}. 
\begin{lemma} \label{a1} For any $m\in \N$
\[
\lim_{n\to\infty}\, \frac{1}{n^m}\int_{[-1,1]^m} \int_{D_m}\, e^{-\sum\limits^m_{i=1} |\sum\limits^{m}_{j=i} x_j|(s_i-s_{i-1})}\, ds\, dx=(2t)^m,
\]
where $D_m=\big\{1<s_1<\cdots<s_m<e^{nt}\big\}$.
\end{lemma}
\begin{proof} Making the change of variables $y_i=\sum\limits^{m}_{j=i} x_j$ for $i=1,2,\dots,m$ gives
\begin{align*}
\frac{1}{n^m}\int_{[-1,1]^m} \int_{D_m}\, e^{-\sum\limits^m_{i=1} |\sum\limits^{m}_{j=i} x_j|(s_i-s_{i-1})}\, ds\, dx
&\leq \frac{1}{n^m}\int_{[-m,m]^m} \int_{D_m}\, e^{-\sum\limits^m_{i=1} |y_i|(s_i-s_{i-1})}\, ds\, dy\\
&= \frac{2^m}{n^m}\int_{D_m}\,   \left(    \prod_{i=1}^m\frac{1-e^{-m(s_i-s_{i-1})}}{s_i-s_{i-1}} \right)\, ds
\end{align*}
and
\begin{align*}
\frac{1}{n^m}\int_{[-1,1]^m} \int_{D_m}\, e^{-\sum\limits^m_{i=1} |\sum\limits^{m}_{j=i} x_j|(s_i-s_{i-1})}\, ds\, dx
&\geq \frac{1}{n^m}\int_{[-\frac{1}{m},\frac{1}{m}]^m} \int_{D_m}\, e^{-\sum\limits^m_{i=1} |y_i|(s_i-s_{i-1})}\, ds\, dy\\
&= \frac{2^m}{n^m}\int_{D_m}\, \bigg(\prod^m_{i=1}\frac{1-e^{-\frac{1}{m}(s_i-s_{i-1})}}{s_i-s_{i-1}}\bigg)\, ds.
\end{align*}
So it suffices to show that, for any $b>0$,
\[
\lim_{n\to\infty} \frac{1}{n^m}\int_{O_{m,3}}\, \Big(\prod^m_{i=1}\, \frac{1-e^{-bu_i}}{u_i}\Big)\, du=t^m,
\]
where $O_{m,3}=\big\{(u_1,u_2,\dots, u_m):\, \sum\limits^m_{i=1} u_i<e^{nt},\, u_1\geq 1,\, u_i> 0,\, i=2,3,\dots, m \big\}$.

Let $O_{m,4}=\big\{(u_1,u_2,\dots, u_m):\,  \sum\limits^m_{i=1} u_i<e^{nt},\, u_i>0,\, i=1,2,\dots , m \big\}$. Then 
\begin{align*}
\int_{O_{m,4}-O_{m,3}}\, \Big(\prod\limits^m_{i=1}\frac{1- e^{-bu_i}}{u_i}\Big)\, du
&\leq \int_{O_{m,5}}\, \Big(\prod\limits^m_{i=1}\frac{1- e^{-bu_i}}{u_i}\Big)\, du,
\end{align*}
where $O_{m,5}=\big\{(u_1,u_2,\dots, u_m):\, 0<u_1\leq 1,\,  0<u_i<e^{nt},\, i=2,\dots ,m \big\}$. 

Note that $h(x)=\frac{1- e^{-bx}}{x}$ is a continuous function on $(0,\infty)$ with $\lim\limits_{x\to 0^+}h(x)=b$ and $\lim\limits _{x\rightarrow \infty} h(x)=0$.  Using the L'H\^opital rule,
\[
\lim_{n\to\infty} \frac{1}{n}\int^{e^{nt}}_0\, h(x)\, dx=t.
\]
This shows
\begin{align*}
\int_{O_{m,4}-O_{m,3}}\, \Big(\prod\limits^m_{i=1}\frac{1- e^{-bu_i}}{u_i}\Big)\, du
&\leq \bigg(\int^1_0\frac{1- e^{-bu}}{u}\, du\bigg) \bigg(\int^{e^{nt}}_0\frac{1- e^{-bu}}{u}\, du\bigg)^{m-1}\leq c_1\, n^{m-1},
\end{align*}
which implies
\[
\lim\limits_{n\to\infty}  \frac{1}{n^m}\int_{O_{m,3}}\, \Big(\prod\limits^m_{i=1}\frac{1- e^{-bu_i}}{u_i}\Big)\, du
=\lim\limits_{n\to\infty}  \frac{1}{n^m}\int_{O_m,4}\, \Big(\prod\limits^m_{i=1}\frac{1- e^{-bu_i}}{u_i}\Big)\, du.
\]
Moreover, 
\begin{align*}
\int_{[0,e^{nt}]^m-O_{m,3}}\, \Big(\prod\limits^m_{i=1}\frac{1- e^{-bu_i}}{u_i}\Big)\, du
&\leq \sum^m_{j=1}\int_{[0,e^{nt}]^m\cap \{u_j>\frac{e^{nt}}{m}\}}\, \Big(\prod\limits^m_{i=1}\frac{1- e^{-bu_i}}{u_i}\Big)\, du\\
&\leq m\, \Big(\int^{e^{nt}}_{\frac{e^{nt}}{m}}\frac{1- e^{-bu}}{u}\, du\Big)\, \Big(\int^{e^{nt}}_0\frac{1- e^{-bu}}{u}\, du\Big)^{m-1}\\
&\leq c_2\, n^{m-1}.
\end{align*}
Therefore,
\begin{align*}
\lim\limits_{n\to\infty}  \frac{1}{n^m}\int_{O_{m,3}}\, \Big(\prod\limits^m_{i=1}\frac{1- e^{-bu_i}}{u_i}\Big)\, du
&=\lim\limits_{n\to\infty}  \frac{1}{n^m}\int_{[0,e^{nt}]^m}\, \Big(\prod\limits^m_{i=1}\frac{1- e^{-bu_i}}{u_i}\Big)\, du\\
&=\lim\limits_{n\to\infty}  \Big(\frac{1}{n}\int^{e^{nt}}_0\, \frac{1- e^{-bu}}{u}\, du\Big)^m\\
&=t^m.
\end{align*}
Combining the above arguments gives the desired result.
\end{proof}

\begin{lemma} \label{a2} For any $m\in \N$,
\begin{align*}
\lim_{n\to\infty} \int_{\R^m}\int_{O_m}  \Big(\prod^{m}_{i=1}| \widehat{f}(y_i)|^2\Big)\, e^{-\sum\limits^{m}_{i=1} |y_i| u_i}\,du\, dy
&=\Big(\int_{\R} |\widehat{f}(y)|^2|y|^{-1}\, dy\Big)^m,
\end{align*}
where $O_m=\big\{(u_1,\dots, u_m): \, u_i>n^{-m},\, i=1\dots,m,\, \sum\limits^{m}_{i=1} u_i<e^{nt}/2\big\}$.
\end{lemma}
\begin{proof} It is easy to see 
\[
\Delta_{m}:=[n^{-m},e^{nt}]^m-O_m\subseteq  \cup^m_{j=1} \Delta_{m,j},
\]
where $\Delta_{m,j}=[n^{-m},e^{nt}]^m\cap \{u_j>\frac{e^{nt}}{m}\}$. Therefore,
\begin{align*}
&\int_{\R^m}\int_{\Delta_{m}}  \Big(\prod^{m}_{i=1}| \widehat{f}(y_i)|^2\Big)\, e^{-\sum\limits^{m}_{i=1} |y_i| u_i}\,du\, dy\\
&\leq  \sum^m_{j=1}\int_{\R^m}\int_{\Delta_{m,j}}  \Big(\prod^{m}_{i=1}| \widehat{f}(y_i)|^2\Big)\, e^{-\sum\limits^{m}_{i=1} |y_i| u_i }\,du\, dy\\
&\leq  c_1 \int_{\R} \int^{e^{nt}}_{\frac{e^{nt}}{m}} | \widehat{f}(y)|^2\, e^{-|y| u}\, du\, dy\\
& = c_1 \int_{\R} | \widehat{f}(y)|^2|y|^{-1}\, \big(e^{-|y|\frac{e^{nt}}{m}}-e^{-|y|e^{nt}}\big)\, du\, dy.
\end{align*}

By the dominated convergence theorem,
\begin{align} \label{a2-1}
\lim_{n\to\infty}\int_{\R^m}\int_{\Delta_m}  \Big(\prod^{m}_{i=1}| \widehat{f}(y_i)|^2\Big)\, e^{-\sum\limits^{m}_{i=1} |y_i| u_i}\,du\, dy=0.
\end{align}

On the other hand,
\begin{align}\label{a2-2} \nonumber
&\lim_{n\to\infty}\int_{\R^m}\int_{[n^{-m},e^{nt}]^m}  \Big(\prod^{m}_{i=1}| \widehat{f}(y_i)|^2\Big)\, e^{-\sum\limits^{m}_{i=1} |y_i| u_i}\,du\, dy\\ \nonumber
&=\lim_{n\to\infty}\Big(\int_{\R}\int^{e^{nt}}_{n^{-m}}  |\widehat{f}(y)|^2\, e^{-|y| u}\,du\, dy\Big)^m\\ 
&=\Big(\int_{\R} |\widehat{f}(y)|^2 |y|^{-1}\, dy\Big)^m
\end{align}

Our result follows easily from \eref{a2-1} and \eref{a2-2}.
\end{proof}

\begin{lemma} \label{a3} For any $m\in \N$
\[
\lim_{n\to\infty}\, \frac{1}{n^m}\int_{O_m}\Big(\prod^m_{i=1}u^{-1}_i\Big)\, du=t^m,
\]
where $O_{m}=\big\{(u_1, \dots, u_m): \, \sum\limits^m_{i=1}u_i<e^{nt},\, u_1\geq 1,\, u_i\geq n^{-m},\, i=2,\dots, m \big\}$.
\end{lemma}
\begin{proof} This result follows from
\begin{align*}
\lim_{n\to\infty}\frac{1}{n^m}\int_{[n^{-m},e^{nt}]^m-O_m}\Big(\prod^m_{i=1}u^{-1}_i\Big)\, du
&\leq \lim_{n\to\infty}\frac{1}{n^m} \sum^m_{j=1}\int_{[n^{-m},e^{nt}]^m\cap \{u_j>\frac{e^{nt}}{m}\}}\Big(\prod^m_{i=1}u^{-1}_i\Big)\, du\\
&\leq c_1\lim_{n\to\infty}\Big(\frac{1}{n}\int^{e^{nt}}_{n^{-m}} u^{-1} du\Big)^{m-1}\Big(\frac{1}{n}\int^{e^{nt}}_{\frac{e^{nt}}{m}} u^{-1} du\Big)\\
&=0
\end{align*}
and
\begin{align*}
\lim_{n\to\infty}\frac{1}{n^m}\int_{[n^{-m},e^{nt}]^m}\Big(\prod^m_{i=1}u^{-1}_i\Big)\, du
&=\lim_{n\to\infty}\Big(\frac{1}{n}\int^{e^{nt}}_{n^{-m}} u^{-1} du\Big)^{m}=t^m.
\end{align*}
\end{proof}

\end{document}